\documentclass{article}

\usepackage[english]{babel}

\usepackage[letterpaper,top=2cm,bottom=2cm,left=3cm,right=3cm,marginparwidth=1.75cm]{geometry}

\usepackage{amsmath}
\usepackage{graphicx}
\usepackage[colorlinks=true, allcolors=blue]{hyperref}
\usepackage{amsthm}
\usepackage{hyperref}

\newtheorem{thm}{Theorem}[section]

\newtheorem{prop}[thm]{Proposition}
\newtheorem{lem}[thm]{Lemma}

\newtheorem{cor}[thm]{Corollary}

\begin{document}
\begin{center}
{\bf {\Large Narayana Sequence and The Brocard-Ramanujan Equation }}
\end{center}

\begin{center}
{\bf Mustafa Ismail}\\
{\bf mostafaesmail$\_$p${@}$sci.asu.edu.eg}\\
{\bf Mathematics Department, Faculty of Science }\\
{\bf Ain Shams University}\\
{\bf  Egypt}
\end{center}

\begin{center}
{\bf Salah Rihanaa \\
salahrihane@hotmail.fr}\\
{\bf Department of Mathematics \\
Institute of Science and Technology}\\
{\bf  University Center of Mila\\
Algeria}
\end{center}
\begin{center}
{\bf M. Anwar\\
mohmmed.anwar@hotmail.com \\
mohamedanwar@sci.asu.edu.eg}\\
{\bf Mathematics Department, Faculty of Science}\\
{\bf Ain Shams University}\\
{\bf Egypt}
\end{center}
\begin{abstract}
Let $\left\lbrace a_{n}\right\rbrace_{n\geq 0}$ be the Narayana Sequence  defined by the recurence  $a_{n}=a_{n-1}+a_{n-3}$ for all $n\geq 3$ with intital values $a_{0}=0$ and $a_{1}=a_{2}=1$. In This paper, we fully characterize the $3-$adic valuation of $a_{n}+1$ and $a_{n}-1$ and then we prove that there are no integer solutions $(u,m)$  to the Brocard-Ramanujan Equation  $m!+1=u^2$ where $u$  is a Narayana number.
\end{abstract}
\section{Introduction}
Diophantine equations involving factorial numbers have been studied by many mathematicians in the last few years. In 1975, Erd{\"o}s and Selfridge \cite{r1} proved $n!$ is a perfect power only when $n=1$. However, one of the most famous among such equations was posed by Brocard \cite{brocard1876question} in 1876 and independently by Ramanujan \cite{ramanujan1962ramanujan} in 1913. The diophantine equation
\begin{equation}\label{eqn1}
    m!+1=u^2
\end{equation}
is now known as \textbf{\textit{Brocard-Ramanujan Diophantine equation}}.
\\ The three known solutions $m=4,5,7$ are easy to check and no further solutions with $m\leq 10^9$ have been proved by Berndt and Galaway in \cite{berndt2000brocard}. Although, Overholt \cite{overholt1993diophantine} showed that the equation (\ref{eqn1}) has only many solutions under a weak version of the abc conjecture. The Brocard-Ramanujan equation is still an open problem. Grossman and Luca \cite{grossman2002sums} showed that if $k$ is fixed, then there are only finitely many positive integers $n$ such that
\begin{equation*}
    F_{n}=m_{1}!+m_{2}!+...+m_{k}!
\end{equation*}
holds for some positive integers $m_{1},m_{2},...,m_{k}$. Also all the solutions for the case $k\leq 2$ were determined. In 1999 Luca \cite{luca1999products} proved that $F_{n}$ is a product of factorials only when $n=1,2,3,6$ and $12 $. Also Luca and stanica \cite{luca2006f} showed that the largest product of distinct Fibonacci numbers which is a product of factorials is $F_{1}F_{2}F_{3}F_{4}F_{5}F_{6}F_{8}F_{10}F_{12}=11!$. In 2012 and In 2016, Marques \cite{marques2012fibonacci} \cite{faco2016tribonacci} proved that $(u,m)=(4,5)$ is the only solution of Eq.(\ref{eqn1}) where $u$ is a Fibonacci number and  there is no solution of Eq.(\ref{eqn1}) when $u$ is a Tribonacci number. Let 
$\left\lbrace a_{n}\right\rbrace_{n\geq 0}$ be the Narayana Sequence  defined by the recurence  $a_{n}=a_{n-1}+a_{n-3}$ for all $n\geq 3$ with intital values $a_{0}=0$ and $a_{1}=a_{2}=1$. The First terms of this sequence are
$$ 0,1,1,1,2,3,4,6,9,28,41,60,88,129,189,277.$$
Some properties of Narayana sequence and its generalizations can be found in \cite{allouche1996narayana}\cite{ballot2017family}\cite{bilgici2016generalized}. In 2021 R.Guadalupe \cite{guadalupe20213} determine all factorials in Narayana sequence. In this paper we solve (\ref{eqn1}) where $u $ is a Narayana number. We are following the same technique that has been used in \cite{marques2012fibonacci}\cite{guadalupe20213} by Vinicius Faco, Diego Marques, Nurettin Irmak and R.Guadalupe. More precisely, we prove the following theorem.  
\begin{thm}\label{thm 1}
There is no integer solution $(m,u)$ for the Brocard-Ramanujan equation (\ref{eqn1}), where $u$ is a Narayana number. 
\end{thm}
\section{PRELIMINARIES}
\begin{lem}
For any integer $m\geq 1$and prime $p$ , we have 
\begin{equation*}
 \frac{m}{p-1}-\left\lfloor\frac{\log m}{\log p}\right\rfloor -1 \leq v_{p}(m!)\leq \frac{m-1}{p-1}.
 \end{equation*}
 \end{lem}
 \begin{proof}
 See \cite{guadalupe20213}.
 \end{proof}
  \begin{lem}
 For all integers $n\geq 1$, we have $\alpha^{n-3} \leq a_{n} \leq \alpha^{n-1}$, where $\alpha$ is the real root of the characteristic polynomial $f(x)=x^{3}-x^{2}-1$ given by
 \begin{equation*}
 \alpha=\frac{1}{3}\left(1+\sqrt[3]{\frac{29-3\sqrt{93}}{2}}+\sqrt[3]{\frac{29+3\sqrt{93}}{2}}\right).
 \end{equation*}
 \end{lem}
 \begin{proof}
 See \cite{guadalupe20213}.
 \end{proof}
\begin{lem}\label{lem2.3}
For all integers $m\geq 3$ and $n\geq 0$ we have 
$$ a_{m+n}=a_{m-1}a_{n+2}+a_{m-3}a_{n+1}+a_{m-2}a_{n}.$$
\end{lem}
\begin{proof}
See \cite{guadalupe20213}.
\end{proof}
 \section{ Lemmata }
\begin{lem}\label{lemma3.1}
\begin{equation*}
v_3(a_i)=\left\{
           \begin{array}{ll}
             0,          & \hbox{$i \equiv 1,2,3,4,6 \mod 8$;} \\
             1,          & \hbox{$i \equiv 5,7,13,15 \mod 24$;} \\
             2,          & \hbox{$i \equiv 8 \mod 24$;} \\
             v_2(i+1)+1, & \hbox{$i \equiv 23 \mod 24$;} \\
             v_2(i+3)+1, & \hbox{$i \equiv 21 \mod 24$;} \\
             v_2(i)+2, & \hbox{$i \equiv 0 \mod 24$;} \\
             v_2(i+8)+2,& \hbox{$i \equiv 0 \mod 24$.} 
           \end{array}
         \right.
\end{equation*}
\end{lem}
\begin{proof}
See \cite{guadalupe20213}.
\end{proof}
\begin{cor} \label{cor3.2}
\hspace{2cm}
\begin{enumerate}
\item If $i\equiv 16,21 \mod 24$, then  $a_{i} \equiv 0 \mod 9$; 
\item If $i\equiv 7 \mod 24$, then $a_{i} \equiv 0 \mod 3$.
\end{enumerate}
\end{cor}
\begin{proof}
The proof is straight forward from the previous lemma \ref{lemma3.1}.
\end{proof}
\begin{prop}\label{prop3.3}
For all integers $s\geq 1 $ and $n\geq 1$ we have
\begin{equation}\label{eqn2} 
\begin{split}
 a_{_{8s3^{n}}}\     &\equiv\ 3^{n+2}\cdot2s \mod 3^{n+3};\\
a_{_{8s3^{n}+1}}\    &\equiv\  3^{n+2}\cdot2s+3^{n+1}\cdot s+1 \mod 3^{n+3};\\
a_{_8s3^{n}+2}\       &\equiv\ 3^{n+2}\cdot2s+1  \mod 3^{n+3}.
\end{split}
\end{equation}
\end{prop}
\begin{proof}
See \cite{guadalupe20213}.
\end{proof}
\begin{prop}\label{prop3.4}
For all integers $s\geq 1 $ and $n\geq 2$ we have 
\begin{equation}\label{eqn3}
\begin{split}
 a_{_{8s3^{n}}}    &\equiv\ 3^{n+3}\cdot2s+3^{n+2}\cdot2s  \mod 3^{n+4};\\
a_{_{8s3^{n}+1}}   & \equiv\ 3^{n+2}\cdot5s+3^{n+1}\cdot s+1  \mod 3^{n+4};\\
a_{_{8s3^{n}+2}}   &\equiv\ 3^{n+3}\cdot2s+3^{n+2}\cdot5s+1  \mod 3^{n+4}.
\end{split}
\end{equation}
\end{prop}
\begin{proof}
 We going to prove this theorem using the principle mathematical induction  on $n$. At $n=2$ we want to prove the following :
\begin{equation}\label{eqn4}
\begin{split}
 a_{_{72s}}    &\equiv\ 3^{4}\cdot8s  \mod 3^{6};\\
a_{_{72s+1}}   & \equiv\ 3^{3}\cdot16s+1  \mod 3^{6};\\
a_{_{72s+2}}   &\equiv\ 3^{4}\cdot11s+1  \mod 3^{6}.
\end{split}
\end{equation}
We can prove them by using the principle of mathematical induction on $s$. At $s = 1$, we have
\begin{eqnarray*}
374009739309 =a_{_{72}}\  &\equiv&   648 \mod 3^{6}; \\ 
548137914373 =a_{_{73}}\  &\equiv&   433 \mod 3^{6};\\
803335158406 =a_{_{74}}\  &\equiv&    163 \mod 3^{6}. 
\end{eqnarray*}
which prove the initial step. Now, Assume that the congruences are true at $s-1$ and we want to prove them at $s$. Using the inductive hypothesis on $s-1$, the defnition of the Narayana numbers and lemma  \ref{lem2.3}, one can deduce the following:
\begin{eqnarray*}
a_{_{72s}} &=& a_{_{72+72(s-1)}}=a_{_{71}}a_{_{72(s-1)+2}}             +a_{_{69}}a_{_{72(s-1)+1}}+a_{_{70}}a_{_{72(s-1)}}\\
          &\equiv& 459 \left( 3^{5}\cdot 2(s-1)+3^{4}\cdot 5(s-1)+1\right)
           + 189 \left(3^{4}\cdot 5(s-1)+3^{3}\cdot (s-1)+1\right)\\
           &&+514 \left(3^{5}\cdot 2(s-1)+3^{4}\cdot 2(s-1)\right) \mod 729\\ 
           &\equiv& 3^{4}\cdot 8s  \mod 729.
\end{eqnarray*}
In the same manner, one can deduce the following:
\begin{eqnarray*}
a_{_{72s+1}}  &\equiv& 3^{3}\cdot16s +1 \mod 729;\\
a_{_{72s+2}} &\equiv& 3^{4}\cdot 11s +1   \mod 729.
\end{eqnarray*}
\\ Thus the congruences (\ref{eqn4}) are true for $s\geq 1$ and $n=2$. Given $s \geq 1$ and $n\geq 2$, assume the congruences (\ref{eqn3}) are true for $n-1$ and we want to prove them at $n$. 
Using the inductive hypothesis and the definition of the Narayana numbers, one can deduce the following:
\begin{eqnarray*}
a_{_{3^{n-1}\cdot8s}}&=&3^{n+2}\cdot2s+3^{n+1}\cdot2s+c_{0}\cdot3^{n+3};\\
a_{_{3^{n-1}\cdot8s+1}}&=&3^{n+1}\cdot5s+3^{n}\cdot s+1+3^{n+3}\cdot c_{1};\\
a_{_{3^{n-1}\cdot8s+2}}&=&3^{n+2}\cdot2s+3^{n+1}\cdot5s+1+3^{n+3}\cdot c_{2};\\
a_{_{3^{n-1}\cdot8s-2}}&=&-3^{n+2}\cdot s+3^{n}\cdot s+1+\left(c_{1}-c_{0}\right)3^{n+3};\\
a_{_{3^{n-1}\cdot8s-1}}&=&3^{n+2}\cdot2s-3^{n}\cdot s+3^{n+3}\left(c_{2}-c_{1}\right).
\end{eqnarray*}
where $c_{0}, c_{1}, c_{2}$ are integers. Using Lemma \ref{lem2.3} and the previous relations, we have
\begin{eqnarray*}
a_{_{2(3^{n-1}\cdot8s)}}&=&a_{_{(3^{n-1}\cdot8s+1)+(3^{n-1}\cdot8s-1)}}\\
&=& a_{_{3^{n-1}\cdot8s}}a_{_{3^{n-1}\cdot8s+1}}+a_{_{3^{n-1}\cdot8s-2}}a_{_{3^{n-1}\cdot8s}}+a_{_{3^{n-1}\cdot8s-1}}a_{_{3^{n-1}\cdot8s-1}}\\
&\equiv & (3^{n+2}\cdot4s+3^{n+3}\cdot2c_{0}+3^{n+1}\cdot4s) \mod 3^{n+4}.\\
\end{eqnarray*}
In the same manner, one can deduce the following:
\begin{eqnarray*}
a_{_{2(3^{n-1}\cdot8s)+1}}&\equiv& 1+3^{n+1}\cdot10s+3^{n}\cdot2s+3^{n+3}\cdot2c_{1} \mod 3^{n+4};\\
a_{_{2(3^{n-1}\cdot8s)+2}}&\equiv & 1+3^{n+2}\cdot4s+3^{n+1}\cdot10s+3^{n+3}\cdot2c_{2} \mod 3^{n+4}.
\end{eqnarray*}
Consequently, 
\begin{eqnarray*}
a_{_{3^{n}\cdot8s}}&=&a_{_{3^{n-1}\cdot8s+2(3^{n-1}\cdot8s)}}=a_{_{3^{n-1}\cdot8s-1}}a_{_{2(3^{n-1}\cdot8s)+2}}+(a_{_{3^{n-1}}\cdot8s}-a_{_{3^{n-1}\cdot8s-1}})a_{_{2(3^{n-1}\cdot8s)+1}}+a_{_{3^{n-1}\cdot8s-2}}a_{_{2(3^{n-1}\cdot8s)}}\\
&\equiv& \left(3^{n+2}\cdot2s-3^{n}\cdot s+\left(c_{2}-c_{1}\right)3^{n+3}\right)\left(1+3^{n+2}\cdot4s+3^{n+1}\cdot10s+3^{n+3}\cdot2c_{2}\right)\\
&&+\left(3^{n+2}\cdot2s+3^{n+1}\cdot2s+c_{0}\cdot3^{n+3}-3^{n+2}\cdot2s+3^{n}\cdot s+\left(c_{1}-c_{2}\right)3^{n+3}\right)\left(1+3^{n+1}\cdot 10s+3^{n}\cdot2s+2c_{1}\cdot3^{n+3}\right)\\
&&+\left(-3^{n+2}\cdot s+3^{n}\cdot s+1+\left(c_{1}-c_{0}\right)3^{n+3}\right)\left(3^{n+2}\cdot4s+2c_{0}\cdot3^{n+3}+3^{n+1}\cdot4s\right) \mod 3^{n+4}\\
&\equiv & 3^{n+3}\cdot 2s+3^{n+2}\cdot 2s \mod 3^{n+4}.
\end{eqnarray*}
In the same manner, one can deduce the following:
\begin{eqnarray*}
a_{_{3^{n}\cdot8s+1}}&\equiv& 3^{n+2}\cdot5s+3^{n+1}\cdot s+1 \mod 3^{n+4};\\
a_{_{3^{n}\cdot8s+2}} &\equiv & 3^{n+2}\cdot5s+3^{n+3}\cdot2s+1 \mod 3^{n+4}.
\end{eqnarray*}
\end{proof}
\begin{thm}\label{thm3.5}
For all $i\geq 1$, we have
\begin{equation*}
  v_{3}(a_{i}-1)= \left\{
                  \begin{array}{ll}
                0 ,    & \hbox{$i\equiv0,4,5,7\mod 8$;}\\
             v_{3}(i-1)+1 ,  &  \hbox{$ i\equiv1 \mod 8$;}\\
             v_{3}(i+2)+1 ,  &  \hbox {$ i\equiv6 \mod 8$;}\\
             v_{3}(i-2)+2 ,   &  \hbox{$  i\equiv2\mod 24$;} \\
             2 ,               &  \hbox{$ i\equiv10 \mod 24$;} \\
             v_{3} (i+6)(i+30)+2 , &  \hbox{$i\equiv18 \mod24$;}\\
             v_{3}(i-3)+2 ,    &\hbox{$ i\equiv3\mod24$;} \\
             v_{3}(i+13)+2 ,   &\hbox{$i\equiv11 \mod 24$;} \\
             v_{3}(i+5)+2 ,   &\hbox{$i\equiv19\mod 24$.} 
             \end{array}
             \right.
\end{equation*}
\end{thm}
\begin{proof}
\underline{Case(1):} $ i \equiv {0,4,5,7} \mod 8$.
\\Subcase(1):  $ i \equiv 0 \mod 8$. We are going to prove that $ v_{3}(a_{i}-1)=0$ using Principle of Mathematical Induction. At $ k=0$, we have $a_0-1 \not\equiv 0 \mod 3$. Now, Assume that $ a_{8k} -1 \not\equiv 0 \mod 3 $ and we want to prove that  $a_{8(k+1)}-1 \not\equiv 0 \mod 3 $. Using lemma \ref{lem2.3}, we have
\begin{eqnarray*}
a_{_{8(k+1)}}-1&=&a_{_{8k+8}}-1=a_{_{7}}a_{_{8k+2}}+a_{_{5}}a_{_{8k+1}}+a_{_{6}}a_{_{8k}}-1\\
 &\equiv& a_{_{8k}}-1 \mod3\\
 &\not\equiv& 0 \mod 3.
 \end{eqnarray*}
 The other subcases can be done in the same way.
\\\underline{Case(2):} $i\equiv 1 \mod 8.$ In this case we have $i-1=3^{n}\cdot 8s$ where $n\geq 1$ and $3\not|\;s $. Using Proposition \ref{prop3.3}, we have 
\begin{eqnarray*}
a_{_{i}}-1&=& a_{_{3^{n} \cdot 8s+1}}-1 \\
&\equiv& 1+3^{n+1} \cdot s+3^{n+2} \cdot 2s-1 \mod 3^{n+3}\\
&\equiv& 3^{n+1} \cdot s \mod 3^{n+3}.
\end{eqnarray*}
Therefore, $v_{3}(a_{i}-1)=n+1=v_{3}(i-1)+1$.
\\\underline{Case(5):}  $ i \equiv 10 \mod 24$. We are going to prove that $ v_{3}(a_{i}-1)=2$ using the principle mathematical induction . At $ k=10$, we have $v_{3}(a_{10}-1)=2$ .
Now, Assume that $  a_{24k+10}-1$  and we want to prove that  $9 ||a_{24(k+1)+10}-1$ . Using Lemma \ref{lem2.3}, we have
 \begin{eqnarray*}
a_{24(k+1)+10}-1 &=& a_{(24k+10)+24}-1=a_{23}a_{24k+12}+a_{21}a_{24k+11}+a_{22}a_{24k+10}-1\\
&\equiv& (a_{24k+10}-1) \mod 9 \\
 &\equiv& 0 \mod 9.
 \end{eqnarray*}
 But, Using Corollary \ref{cor3.2}
 \begin{eqnarray*}
 a_{24(k+1)+10}-1 &\equiv& 9\left(a_{24k+12}+a_{24k+12}+a_{24k+11}+a_{24k+10}\right)+a_{24k+10}-1 \mod 27 \\
 &\equiv& 9(3a_{24k+11}+3a_{24k+9}+a_{24k+7})+a_{24k+10}-1 \mod 27 \\
 &\equiv& a_{24k+10}-1 \mod 27\\
 &\not \equiv & 0 \mod 27.
 \end{eqnarray*}
 Therefore, $v_{3}(a_{i}-1)=2.$
\\\underline{Case(6):} $ i \equiv 18\mod 24$. We want to prove that:
\begin{eqnarray*}
v_{3}(a_{24k+18}-1)&=&v_{3}\left((24k+24)(24k+48)\right)+2\\
&=& v_{3}\left(24^2 (k+1)(k+2)\right)+2\\
&=&v_{3}\left((k+1)(k+2)\right)+4.
\end{eqnarray*}
\\Subcase(1):  $k\equiv 0 \mod 3 .$ We are going to prove that $ v_{3}(a_{24k+18}-1)=4$ using the principle mathematical induction . At $ k=0$, we have $a_{18}-1 \equiv 0 \mod 81$ and $a_{18}-1 \not\equiv 0 \mod 243$. Now, Assume that $ a_{72k+18}-1 \equiv 0 \mod 81$ and $a_{72k+18}-1 \not\equiv 0 \mod 243 $ and we want to prove that $a_{72(k+1)+18}-1 \equiv 0 \mod 81$ and $a_{72(k+1)+18}-1 \not\equiv 0 \mod 243  .$ Using lemma \ref{lem2.3}, we have
\begin{eqnarray*}
a_{72(k+1)+18}-1&=&a_{(72k+18)+72}-1=a_{71}a_{72k+20}+a_{69}a_{72k+19}+a_{70}a_{72k+18}-1 \\
&\equiv& 27(2a_{72k+20}+a_{72k+19}+a_{72k+18})+a_{72k+18}-1 \mod 81\\
&\equiv& a_{72k+18}-1 \mod 81 \\
&\equiv& 0 \mod 81.
\end{eqnarray*}
and  
\begin{eqnarray*}
a_{72(k+1)+18}-1&=&27\left(8a_{72k+20}+7a_{72k+19}+a_{72k+18}\right)+a_{72k+18}-1 \mod 243\\
&=&27\left(9a_{72k+20}-a_{72k+20}+7a_{72k+19}+a_{72k+18}\right)\\
&=&27\left(9a_{72k+20}+7a_{72k+16}+9a_{72k+18}-a_{72k+21}\right)+a_{72k+18}-1\mod 243\\
 &\equiv& a_{72k+18}-1 \mod 243\\
 &\not \equiv& 0 \mod 243.
\end{eqnarray*}
 Therefore, $v_{3}(a_{i}-1)=4.$
\\Subcase(2): $k\equiv 1 \mod 3. $
In this case we have $i=3^{n}\cdot 8s-30$ where $n\geq 2$ and $3\not|\;s$. Using lemma \ref{lem2.3} and proposition \ref{prop3.4} then we have 
\begin{eqnarray*}
a_{i}-1&=&a_{3^{n}\cdot8s-30}-1=a_{3^{n}\cdot8s-27}-a_{3^{n}\cdot8s-28}-1\\
&=&a_{3^{n}\cdot8s-24}-2a_{3^{n}\cdot8s-25}+a_{3^{n}\cdot8s-26}-1\\
&=&-3a_{3^{n}\cdot8s-21}-2a_{3^{n}\cdot8s-22}+3a_{3^{n}\cdot8s-20}-1\\
&=&-8a_{3^{n}\cdot8s-18}+4a_{3^{n}\cdot8s-17}+a_{3^{n}\cdot8s-16}-1\\
&=&-12a_{3^{n}\cdot8s-12}-26a_{3^{n}\cdot8s-11}+21a_{3^{n}\cdot8s-10}-1\\
&=&73a_{3^{n}\cdot8s-6}-63a_{3^{n}\cdot8s-5}+9a_{3^{n}\cdot8s-4}-1\\
&=&-64a_{3^{n}\cdot8s-4}+136a_{3^{n}\cdot8s-3}-63a_{3^{n}\cdot8s-2}-1\\
&=&a_{3^{n}\cdot8s-2}-200a_{3^{n}\cdot8s-1}+136a_{3^{n}\cdot8s}-1\\
&=&201a_{3^{n}\cdot8s+1}-200a_{3^{n}\cdot8s+2}+135a_{3^{n}\cdot8s}-1\\
&\equiv& 201 \left(3^{n+2}\cdot5s+3^{n+1}\cdot s+1\right)-200\left(3^{n+3}\cdot2s+3^{n+2}\cdot 5s+1\right)\\
&+&135\left(3^{n+3}2s+3^{n+2}\cdot2s\right)-1\mod 3^{n+4}\equiv -3^{n+3}\cdot130s \mod 3^{n+4}.
\end{eqnarray*}
\\Therfore, $v_{3}(a_{i}-1)=n+3=v_{3}(i+30)+3.$
\\Subcase(3) of case(5) and the cases(3),(4),(7), (8) and(9) can be done in the same way.
\end{proof}
\begin{thm}\label{thm3.6}
For all $i\geq 1$, we have
\begin{equation*}
  v_{3}(a_{i}+1)= \left\{
                  \begin{array}{ll}
                0 ,    & \hbox{$i\equiv0,1,2,3,5,6,7\mod 8$;}\\
                1 ,  &  \hbox{$ i\equiv4,12 \mod 24$;}\\
             v_{3}(i+4)+1  ,  &  \hbox {$ i\equiv20\mod 24$.}
             \end{array}
             \right.
\end{equation*}
\end{thm}
\begin{proof}
\underline{Case(1):}  $ i \equiv {0,1,2,3,5,6,7}\mod 8$.
\\Subcase(1): $ i \equiv 0 \mod 8$. We are going to prove that $ v_{3}(a_{i}+1)=0$ using the principle mathematical induction. At $ k=0$, we have $a_{8}+1 \not\equiv 0 \mod 3.$ Now, Assume that $ a_{8k} +1 \not\equiv 0 \mod  3 $ and we want to prove that $a_{8(k+1)}+1 \not\equiv 0 \mod  3 .$ Using lemma \ref{lem2.3}, we have
\begin{eqnarray*}
a_{8(k+1)}+1&=&a_{8k+8}+1=a_{7}a_{8k+2}+a_{5}a_{8k+1}+a_{6}a_{8k}+1\\
&\equiv& (a_{8k}+1) \mod 3 \not\equiv 0 \mod 3.
\end{eqnarray*}
 Therefore, $v_{3}(a_{i}+1)=0$ .
\\The all other subcases can be done in the same way.
\\\underline{ Case(2):} $i\equiv 4,12 \mod 24$
\\Subcase(1):   $ i \equiv 4 \mod 24$. We are going to  prove that $ v_{3}(a_{i}+1)=1$ using the principle mathematical induction. At $ k=4$, we have $a_4+1 \equiv 0 \mod 3$ and  $a_4+1 \not\equiv 0 \mod 9.$ Now, Assume that $ a_{24k+4} +1 \equiv 0 \mod 3$ and $ a_{24k+4} +1 \not\equiv 0 \mod  9 $ and we want to prove that $a_{24(k+1)+4}+1 \equiv 0 \mod 3$  and  $a_{24(k+1)+4}+1 \not\equiv 0 \mod 9 .$ Using lemma \ref{lem2.3}, we have
\begin{eqnarray*}
a_{24(k+1)+4}+1&=&a_{24k+24+4}+1=a_{23}a_{24k+6}+a_{21}a_{24k+5}+a_{22}a_{24k+4}+1\\
&\equiv& (a_{24k+4}+1) \mod 9 \equiv 0 \mod 3  \not\equiv 0 \mod 9.
\end{eqnarray*}
Therefore, $v_{3}(a_{i}+1)=1$ .
\\Sub case(2) can be done in the same way.
\\\underline{Case(3):}  $i\equiv 20 \mod 24  .$ In this case we have $i=3^{n}\cdot 8s-4$ where $n\geq 1 $ and $3\not| s\;.$ Using lemma \ref{lem2.3} and proposition \ref{prop3.3}. Then, we have  
\begin{eqnarray*}
a_{i}+1&=&a_{8s3^{n}-4}+1=a_{8s3^{n}-1}-a_{8s3^{n}-2}+1\\
&=&a_{8s3^{n}+2}-2a_{8s3^{n}+1}+a_{8s3^{n}}+1 \\
&\equiv& -3^{n+1}\cdot2s  \mod 3^{n+3}.
\end{eqnarray*}
Therefore, $v_{3}(a_{i}+1)=n+1=v_{3}(i+4)+1=n+1.$
\end{proof}
\section{Proof of Theorem \ref{thm 1}}
\begin{proof}
If $n\geq 3$ then no solution for equation \ref{eqn1}. Now Suppose that $n > 3$ and we use the fact 
\begin{equation*}
 \frac{m}{2}-\left\lfloor\frac{\log m}{\log 3}\right\rfloor -1 \leq v_{3}(m!);
 \end{equation*}
 together with theorem \ref{thm3.5} and theorem \ref{thm3.6}. We get, 
\begin{equation*}
 \frac{m}{2}-\left\lfloor\frac{\log m}{\log 3}\right\rfloor -1 \leq v_{3}(m!)=v_{3}(a_{n}-1)+v_{3}(a_{n}+1) \leq v_{3}((n-1)(n+2)(n-2)(n+6)(n+30)(n-3)(n+13)(n+15)(n+4)+16;
 \end{equation*}
 Thus,
\begin{equation*}
\frac{m}{2}-\left\lfloor\frac{\log m}{\log 3}\right\rfloor -1 \leq 9 v_{3}(n+w)+16;
\end {equation*}
where $w \in\left\lbrace-1,2,-2,6,30,-3,13,5,4\right\rbrace$.
Therefore,
\begin{equation*}
3^{\left\lfloor\frac{1}{9}\left(\frac{m}{2}-\left\lfloor\frac{\log m}{\log 3}\right\rfloor -17\right)\right\rfloor}\leq n+w\leq n+30;
\end{equation*} 
By applying the log function, we obtain 
\begin{equation}\label{eqn5}
   \left\lfloor\frac{1}{9}\left(\frac{m}{2}-\left\lfloor\frac{\log m}{\log 3}\right\rfloor-17\right)\right\rfloor \leq \frac{n+30}{\log 3}. 
\end{equation}
On the other hand, $$(1.64)^{2n-6} \leq a_{n}^{2}=m!+1<2\left(\frac{m}{2}\right)^{m};$$ So
\begin{equation*}
    n < 4+(1.33) m \log\left(\frac{m}{2}\right);
\end{equation*}
Substituting in equation \ref{eqn5}, we obtain 
\begin{equation*}
   \left\lfloor\frac{1}{9}\left(\frac{m}{2}-\left\lfloor\frac{\log m}{\log 3}\right\rfloor-17 \right)\right\rfloor \leq \frac{34+1.33 \log\left(\frac{m}{2}\right)}{\log 3}.  
   \end{equation*}
This inequality yields $m \leq 221$. Then $n\leq 1386$. Now, we use a simple routine written in sage which does not return any solution in the range $n \leq 1386$. 
The proof is completed .
\end{proof}
\bibliographystyle{plain}
\bibliography{References}

\end{document}